\documentclass[a4paper]{scrartcl}

\usepackage{amsfonts,amsthm,amssymb,amsmath}
\usepackage[latin1]{inputenc}
\usepackage[T1]{fontenc}
\usepackage{graphicx}
\usepackage[inline]{enumitem}
\usepackage{hyperref}
\usepackage{tikz}
\usepackage{tkz-berge}

\newtheorem{thm}{Theorem}
\newtheorem{lem}[thm]{Lemma}

\newtheorem{cor}[thm]{Corollary}

\theoremstyle{definition}

\newcommand{\id}{\mathrm{id}}

\newcommand{\br}{\operatorname{br}}
\newcommand{\gr}{\operatorname{gr}}

\title{Firefighting on trees and Cayley graphs}
\author{Florian Lehner \thanks{The author was supported by the Austrian Science Fund (FWF) Grant no. J 3850-N32}}

\begin{document}
\maketitle

\begin{abstract}
We study Hartnell's firefighter problem on infinite trees and characterise the branching number in terms of the firefighting game. Using our results about trees, we give a partial answer to a question of Mart\'inez-Pedroza concerning firefighting on Cayley graphs.
\end{abstract}

\section{Introduction}

In 1995 Hartnell~\cite{hartnelltalk} introduced the firefighting game which can be described as follows. Before the first round of the game, an antagonist sets some subset of the vertices of a graph $G$ on fire. Then, in each round $n$, we can protect $f_n$ vertices whereafter the fire spreads to all unprotected neighbours of burning vertices. Once a vertex is burning or protected, it remains in that state for the rest of the game. This can for example be seen as a model for the spread of a perfectly contagious disease with no cure, see~\cite{grids-devlinhartke}. The act of protecting vertices at each time step, could then be viewed as vaccinations.

There are several different goals that we might want to pursue, e.g.\ minimise number of rounds or number of burnt vertices, or save a certain set of vertices or a given fraction of the vertices from being burnt.  The survey paper~\cite{ffsurvey} gives an overview on different lines of research concerning the firefighting game.

In this paper we will focus on the question of containment. We say that a fire can be contained on an infinite graph if we can prevent it from spreading to infinitely many vertices. An infinite graph $G$ satisfies $f_n$-containment, if any finite initial fire can be contained by protecting $f_n$ vertices in round $n$.

Containment was first studied in grids, the first results being that certain planar grids satisfy constant containment, i.e. containment for $f_n \equiv c$, see~\cite{grids-devlinhartke,3gridmessinger}. Develin and Hartke~\cite{grids-devlinhartke} showed that higher dimensional square grids do not satisfy constant containment. However, it is easy to see that they satisfy $f_n$-containment for some polynomial $f_n$. In fact, Dyer, Mart\'inez-Pedroza, and Thorne~\cite{coarsegeometry} showed that every graph with polynomial growth of degree $d$ has the $c\cdot n^{d-2}$-containment property for some constant $c$.

We study the question of exponential containment. We say that a graph satisfies exponential containment of rate $\lambda$ if it satisfies $f_n$-containment for some $f_n = O(\lambda^n)$. It is easy to see that for every graph $G$ there is a threshold $\lambda_c$ such that for every $\lambda > \lambda_c$ it satisfies $\lambda$-containment whereas for $\lambda < \lambda_c$ it doesn't.

We prove that if $T$ is a tree, then $\lambda_c$ coincides with the branching number $\br T$ of this tree (see the next section for a definition). It is worth noting that the branching number also shows up as a threshold in different problems. It marks the transition from transience to recurrence of the homesick random walk on a tree and $\frac{1}{\br T}$ is the percolation threshold on an infinite tree, see~\cite{branchingnumberlyons}. This naturally leads to the question if his is coincidence or there is a deeper connection between firefighting and random processes on graphs.

As an application of our results about trees we make progress towards a question of Mart\'inez-Pedroza~\cite{groups}. He showed that Cayley graphs of non-amenable groups do not have the polynomial containment property and asked whether polynomial containment always implies polynomial growth for Cayley graphs. We show that for a Cayley graph with exponential growth of rate $\alpha$ we have $\lambda_c = \alpha$. This implies that Cayley graphs of exponential growth can never satisfy polynomial containment, only leaving open the notoriously difficult case of groups with intermediate growth. 

\section{Preliminaries}
\label{sec:prelim}

Throughout this paper $G = (V,E)$ denotes a graph with vertex set $V$ and edge set $E$. All graphs considered will be connected and locally finite.

The firefighting game is defined as follows: Let $G$ be an infinite graph and let $(f_n)_{n \in \mathbb N}$ be a sequence of integers. Before the first round, a finite set $X_0$ of vertices of $G$ are defined as burning. In round $n$, the player can pick $f_n$ vertices which are not burning to mark as protected. Afterwards every unprotected vertex which is adjacent to a burning vertex is marked as burning. Note that once a vertex is marked as burning or protected, it remains in that state until the end of the game.

The player wins the game, if after finitely many rounds no new vertices are marked as burning---in this case we say that the fire is \emph{contained}. A \emph{containment strategy} for an initial fire $X_0$ is a map $s$ from $\mathbb N$ to the power set of $V$, where $|s(n)|\leq f_n$ such that marking all vertices in $s(n)$ as protected in round $n$ leads to containment.

A graph satisfies \emph{$f_n$-containment}, if there is a containment strategy (with respect to the sequence $f_n$) for any initial fire $X_0$. A graph $G$ satisfies \emph{exponential containment of rate $\lambda$} if there is $f_n = O(\lambda^n)$ such that $G$ satisfies $f_n$-containment. Clearly, if $G$ satisfies exponential containment of rate $\lambda$, then it also satisfies exponential containment of any rate $\lambda' > \lambda$. Hence there is a critical rate $\lambda_c$ such that for $\lambda < \lambda_c$, the graph $G$ does not satisfy exponential containment of rate $\lambda$, whereas for $\lambda > \lambda_c$ it does.

Let $r \in V$ and assume that $G$ is rooted at $r$. For a vertex or edge $x$ denote by $|x|$ the length of a shortest path containing both $r$ and $x$. Define the \emph{ball of radius $k$ with center $r$} by $B_r(k) = \{v \in V \mid |v| \leq k\}$. 

The \emph{(exponential) growth rate} of a graph is defined by $\gr G = \lim _{k \to \infty} \left( |B_r(k)| \right)$ if the limit exists. Note that if the growth rate exists, then it does not depend on the base point $r$.

For a tree $T$ the \emph{branching number} $\br T$ provides another measure for growth. It was first studied by Furstenberg~\cite{branchingnumberfurstenberg}, and later formally defined by Lyons~\cite{branchingnumberlyons} who pointed out its close connections to random walks and percolation on trees. For a tree $T$ rooted at $r$ define
\[
\br T := \sup \left\{\lambda \mid \exists \text{ non-zero flow }\theta \text{ from } r \text{ to }\infty\text{ s.t.\ } \theta(e)= \lambda^{-|e|}\right\}.
\]

By a variant of the well known max-flow min-cut theorem we get the following equivalent definition:
\[
	\br T = \sup \left\{ \lambda \mid \inf_{\Pi} \sum _{e \in \Pi} \lambda ^{-|e|} > 0\right\},
\]
where the infimum runs over all cutsets $\Pi \subseteq E$ whose removal leaves the root $r$ in a finite component.

\section{Trees}
\label{sec:trees}

In this section we determine the critical rate $\lambda_c$ for exponential containment on trees. It turns out that $\lambda_c$ equals the branching number. Hence, our main theorem can be used to define the branching number in terms of the firefighter game. We first prove two lemmas which tell us that in order to show containment for a tree $T$ it suffices to study a very restricted set of strategies.

\begin{lem}
Let $T$ be a tree rooted at $r$. Then $T$ satisfies $f_n$-containment if and only if there is a containment strategy for each $B_r(k), k \in \mathbb N$.
\end{lem}

\begin{proof}
The forward direction is trivial: if there is a winning strategy for every finite set, then there is a winning strategy for every $B_r(k)$. Conversely, if $X_0$ is any finite set, then there is some $B_r(k)$ such that $X_0 \subseteq B_r(k)$. So the vertices on fire at step $n$ for starting set $B_r(k)$ is a subset of the vertices on fire for starting set $X_0$. Hence a winning strategy for $B_r(k)$  is also winning for $X_0$.
\end{proof}

Let $V' \subseteq V$ be a finite set of vertices. We can define a strategy $s(V')$ by $n \mapsto S_n$, where $S_n$ is the set containing the $f_n$ vertices in $V'$ that are closest to the root and neither burning nor protected.

\begin{lem}
Let $T$ be a tree rooted at $r$ and assume that the set of vertices initially on fire is $B_r(k)$. If there is a containment strategy, then there is a containment strategy of the form $s(V')$.
\end{lem}

\begin{proof}
Let $s$ be any successful containment strategy for starting set $B_r(k)$. Let $X_F$ be the final set of vertices on fire after the successful containment strategy is played. Let $V'$ be the set of vertices in $V \setminus X_F$ which have a neighbour in $X_F$ and denote by $s'$ the strategy defined as above. Note that since $T$ is a tree, $V'$ contains exactly one vertex on every ray starting at $r$. Now if a vertex $v_0 \in V'$ was on fire before it is played in $s'$, then
\[
	|\{v \in V' \colon |v| < n\}| > \sum _{i=1}^{n-k} f_i
\]
where $n = |v_0|$. This means that $s$ can't protect all vertices in $\{v \in V' \colon |v| < n\}$ before step $n$. But then $X_F$ can't be the set of vertices on fire after $s$ is played: since $T$ is a tree, all vertices at distance $n$ are on fire after $k-n$ steps, unless a vertex of the unique path from $r$ has been played before.
\end{proof}

We are now ready to prove the main theorem of this paper.

\begin{thm}
\label{thm:trees}
If $T$ is an infinite, locally finite tree, then $\lambda_c =\br T$.
\end{thm}

\begin{proof}
We first show that $\lambda_c \leq \br T$. Clearly it suffices to show that for every $\lambda > \br T$ there is a successful containment strategy with $f_n = \lfloor \lambda ^n \rfloor$. Hence let $\lambda > \br T$ and assume that the starting set is $B_r(k)$. Since
\[
\inf_\Pi \sum_{e \in \Pi} \lambda ^{-|e|} = 0,
\]
we can pick a cutset $\Pi$  whose removal leaves $r$ in a finite component such that 
\[
\sum_{e \in \Pi} \lambda ^{-|e|} < \epsilon,
\]
where $\epsilon$ is chosen in a way that $\epsilon \cdot  \lambda ^n  < \lfloor \lambda^{n-k} \rfloor$.
Let $V'$ be the set containing the endpoint of each $e \in \Pi$ which is further away from $r$. Let $V_n' := \{v \in V' \colon |v| = n\}$. Then
\[
	\epsilon \geq \sum_{e \in \Pi} \lambda ^{-|e|} = \sum_{v \in V'} \lambda ^{-|v|} > \sum_{v \in V_n'} \lambda ^{-|v|} = |V_n'| \cdot \lambda ^{-n},
\]
whence
\[
	|V_n'| \leq \epsilon \cdot \lambda ^n < \lfloor \lambda^{n-k}\rfloor.
\]
This implies that we can play the set $V_n'$ at step $n-k$ (i.e.\ before the fire reaches level $n$). Hence the fire is contained below $V'$ and the strategy is successful.

To show that $\lambda_c \geq \br T$ it suffices to show that there is no containment strategy for $f_n = \lfloor \lambda^n \rfloor$ for any $\lambda < \br T$. Indeed, this implies that there is no containment strategy for $f_n = o(\lambda^n)$, and since we can choose $\lambda$ arbitrarily close to $\br T$, it follows that $\lambda_c \geq \br T$.

Hence let $\lambda < \br T$. Let $C$ be a constant such that 
\[
\sum_{i=1}^n \lfloor \lambda^i \rfloor \leq C \cdot \lambda ^n.  
\]
Choose $\mu$ such that $\lambda < \mu < \br T$. Note that since $\mu < \br T$, there is some $\epsilon > 0$ such that for every cutset $\Pi$ we have 
\[
	\sum_{e \in \Pi} \mu^{-|e|} > \epsilon.
\]
Finally let $k$ be such that
\[
	C \cdot \sum_{i=k+1}^\infty \left( \frac \lambda \mu \right)^i < \epsilon.
\]

We now claim that with $k$ chosen as above, there is no successful containment strategy for $X_0 = B_r(k)$. Assume there was one, then there is one of the form $s(V')$. Since $s(V')$ is assumed to be a containment strategy, removing $V'$ from $G$ leaves $r$ in a finite component. Let $\Pi \subseteq E$ be the set containing for every $v' \in V'$ the first edge of the path from $v'$ to $r$. Then $\Pi$ is a cutset whose removal leaves $r$ in a finite component. 

Let $V_n' :=\{v \in V' \colon |v| = n\}$. Since our strategy is successful we know that every vertex in $V'$ is played before it catches fire. In particular
\[
	|V_n'| \leq \sum _{i=1}^{n-k}\lfloor \lambda^i \rfloor \leq C \cdot \lambda ^n.
\]
Furthermore $|V_n'| = 0$ for $n \leq k$ because we can never play any vertex which is initially on fire. Putting all of the above together we get
\[
	\epsilon 
    < \sum_{e \in \Pi} \mu^{-|e|} 
    = \sum_{v \in V'} \mu ^{-|v|} 
    = \sum_{i=k+1}^{\infty} \sum_{v \in V_i} \mu^{-i} 
    \leq \sum_{i=k+1}^{\infty} C \cdot \lfloor \lambda ^i \rfloor \cdot \mu^{-i}  
    \leq C \cdot \sum_{i=k+1}^\infty \left( \frac \lambda \mu \right)^i 	< \epsilon,
\]
which is a contradiction.
\end{proof}

\section{Cayley graphs}
\label{sec:cayley}

In this section we use the main result of the previous section as well as some known results about Cayley graphs to determine the exponential containment threshold $\lambda_c$ for Cayley graphs.

For this purpose we need the following definition. Let $T$ be a tree rooted at $r$. For a vertex $v$ define $T_v$ to be the subtree of $T$ rooted at $v$, i.e.\ let $e$ be the first edge of the unique path from $v$ to $r$, then $T_v$ is the component of $T - e$ which contains $v$ (rooted at $v$). The tree is called subperiodic, if there is $k \in \mathbb N$ such that for every $v$ there is $v'$ with $|v'| \leq k$ and $T_v$ embeds into $T_{v'}$ as a subtree in a way that maps $v$ to $v'$.

Furstenberg~\cite{subperiodicfurstenberg} showed that for a subperiodic tree the growth rate exists and coincides with the branching number, see \cite{probtreeslyonsperes} for a graph theoretic proof.

Let $\Gamma$ be a group and let $G$ be a Cayley graph of $\Gamma$ with respect to the generating set $\{x_1,\dots, x_k\}$. The following construction due to Lyons~\cite{growthrwlyons} gives a subperiodic spanning tree of $G$ with the same exponential growth rate as $G$:
For every $v \in \Gamma$ there is a unique word $[v] = (x_{i_1},\dots,x_{i_l})$ such that 
\begin{itemize}[label=-]
\item $x_{i_1}\cdots x_{i_l} = v$
\item $l$ is the distance from $v$ to $\id$ in $G$, and 
\item $[v]$ is lexicographically minimal among all words with the first two properties.
\end{itemize}
Now the graph with vertex set $\Gamma$ and an edge from $v$ to $w$ if $[v]$ is an extension of $[w]$ by one letter (or vice versa) is easily seen to be a subperiodic spanning tree, rooted at $\id$.

From this we can now deduce the following result.
\begin{thm}
Let $G$ be a Cayley graph of a group with exponential growth rate $\alpha$. Then
\begin{itemize}
\item $G$ satisfies exponential containment of any rate $\lambda > \alpha$,
\item $G$ does not satisfy exponential containment of any rate $\lambda < \alpha$.
\end{itemize}
\end{thm}
\begin{proof}
For the first part consider the following strategy: Note that if the fire initially is contained in a ball with radius $k$ about some vertex, then at step $n$ it will be contained in a ball of radius $k+n$. Wait until $\lambda^n$ is larger than the boundary of this ball, then play all vertices in this boundary at once. This is possible since $\lambda^n$ asymptotically grows quicker than $|B_r(k+n+1)|$ and hence also faster than the boundary of the ball of radius $k+n$.

For the second part note that if $G$ satisfies exponential containment of some rate $\lambda$, then so does every subgraph of $G$. But the subperiodic spanning tree of $G$ with exponential growth rate $\alpha$ does not satisfy exponential containment of any rate $\lambda < \alpha$ by Theorem~\ref{thm:trees}.
\end{proof}

\begin{cor}
A Cayley graph of a group with exponential growth never satisfies polynomial containment.
\end{cor}

\begin{proof}
For any $d,\lambda \in \mathbb R$ we have $n^d = o(\lambda^n)$.
\end{proof}
\bibliographystyle{abbrv}
\bibliography{ref.bib}

\end{document}